\documentclass[10pt]{article}
\usepackage{times}
\usepackage{natbib}	% 
\usepackage{amsmath}	% 
\usepackage{amsthm}		% 
\usepackage{amssymb}	% 
\usepackage{graphicx}	% 
\usepackage{url}
\usepackage{hyperref}

\newcommand{\0}{{\bm 0}}

\newcommand{\A}{\bm{A}}

\newcommand{\BigZero}{\mbox{\Large \bf 0}}

\newcommand{\D}{\bm{D}}

\newcommand{\Fc}{{\cal F}}
\newcommand{\F}{\bm{F}}

\newcommand{\I}{\bm{I}}

\newcommand{\M}{\bm{M}}

\newcommand{\Pc}{{\cal P}}

\newcommand{\pfe}[1]{{\spab \! \left( {{\textstyle #1}} \right)} }

\renewcommand{\P}{\bm{P}}

\newcommand{\Realpart}{\mathrm{Re}}

\newcommand{\alignit}[1]{\begin{align}#1\end{align}}
\newcommand{\aligno}[1]{\begin{align*}#1\end{align*}}

\newcommand{\alphav}{\bms{\alpha}}

\newcommand{\bdes}{\begin{description}}
\newcommand{\bd}{\begin{description}}
\newcommand{\bean}{\begin{eqnarray*}}
\newcommand{\benu}{\begin{enumerate}}
\newcommand{\ben}{\begin{enumerate}}

\newcommand{\bite}{\begin{itemize}}
\newcommand{\bi}{\begin{itemize}}
\newcommand{\bms}[1]{{\boldsymbol{#1}}}
\newcommand{\bmu}{\begin{multline}}
\newcommand{\bm}[1]{{\bf #1}}

\newcommand{\df}[2]{\displaystyle \frac{\mbox{\rm d} #1}{\mbox{\rm d} #2}}

\newcommand{\diag}[1]{\mbox{\bf diag}\matrx{#1}}

\newcommand{\dspfrac}[2]{\frac{\displaystyle #1}{\displaystyle #2} }

\newcommand{\eban}{\begin{eqnarray*}}
\newcommand{\eba}{\begin{eqnarray}}
\newcommand{\eb}{\begin{equation}}
\newcommand{\edes}{\end{description}}

\newcommand{\ed}{\end{description}}
\newcommand{\eean}{\end{eqnarray*}}
\newcommand{\eea}{\end{eqnarray}}
\newcommand{\eenu}{\end{enumerate}}
\newcommand{\een}{\end{enumerate}}
\newcommand{\ee}{\end{equation}}
\newcommand{\eite}{\end{itemize}}
\newcommand{\ei}{\end{itemize}}
\newcommand{\emu}{\end{multline}}

\newcommand{\eqdef}{:=}

\newcommand{\ev}{\bm{e}}

\newcommand{\goesto}{\rightarrow}

\newcommand{\hide}[1]{}	%hides text
\newcommand{\matrx}[1]{{\left[ \stackrel{}{#1}\right]}}

\newcommand{\pf}[2]{\displaystyle \frac{\partial #1}{\partial #2}}
\newcommand{\ppf}[2]{\displaystyle \frac{\partial^2 #1}{\partial #2^2}}

\newcommand{\pv}{\bm{p}}
\newcommand{\p}{\bm{p}}

\newcommand{\spab}{\, r}	% spectral abscissa 

\newcommand{\suchthat}{\colon}

\newcommand{\tp}{\top}	%rm better than \sf, \tt, \rm T, or \textsf{\textsc{t}}
\newcommand{\union}{\cup}

\newcommand{\uv}{\bm{u}}

\renewcommand{\v}{\bm{v}}

\newcommand{\vv}{\bm{v}}

\newcommand{\x}{\bm{x}}

\newcommand{\y}{\bm{y}}

\newcommand{\z}{\bm{z}}
\newfont{\gilfont}{cmsy10 scaled\magstep0}
\newcommand{\Reals}{\mathbb{R}} % reals
\newcommand{\uvf}{\uv}
\newcommand{\vvf}{\vv}
\newcommand{\Grm}{\D}
\newcommand{\Gr}{D}
\newcommand{\Mvm}{\A}
\newcommand{\Am}{\F}

\newtheorem{Theorem}{Theorem:}
\newtheorem{Corollary}[Theorem]{Corollary:}
\newtheorem{Lemma}[Theorem]{Lemma:}

\theoremstyle{plain}
\newtheorem{Main}{Main Result\!\!}

\newtheorem{Definitions}{Definitions\!}

\newtheorem{TheoremFiveTwo}{Theorem 5.2\!}

\begin{document}
\title{Karlin Theory
On Growth and Mixing\\
Extended to Linear Differential Equations}
\author{Lee Altenberg \\ \url{altenber@hawaii.edu}}

\maketitle
%\sloppy

\begin{abstract}
Karlin's (1982) Theorem 5.2 shows that linear systems alternating between growth and mixing phases have lower asymptotic growth with greater mixing.  Here this result is extended to linear differential equations that combine site-specific growth or decay rates, and mixing between sites, showing that the spectral abscissa of a matrix $\D + m \A$ decreases with $m$, where  $\D \neq c \I$ is a real diagonal matrix, $\A$ is an irreducible matrix with non-negative off-diagonal elements (an ML- or essentially non-negative matrix), and $m \geq 0$. The result is based on the inequality: $\uv^\tp \A \vv < \pfe{\A}$, where $\uv$ and $\vv$ are the left and right Perron vectors of the matrix $\D + \A$, and $\pfe{\A}$ is the spectral abscissa and Perron root of $\A$.  The result gives an analytic solution to prior work that relied on two-site or numerical simulation of models of growth and mixing, such as source and sink ecological models, or multiple tissue compartment models of microbe growth.  The result has applications to the Lyapunov stability of perturbations in nonlinear systems. 
\end{abstract}

\section{Introduction}

Growth and movement are ubiquitous phenomena in physical, biological, and social systems.  In particular, site-specific growth and decay rates, combined with movement between sites, can be found throughout nature.  An even wider range of phenomena may be included if we consider the formal equivalent:  \emph{state}-specific growth rates, and transformations between states.

\citet{Karlin:1982} developed two very general theorems on the asymptotic growth rates of systems combining growth and movement.  The context that motivated these developments was rather narrow:  analysis of the protection of genetic diversity in a subdivided population undergoing natural selection and migration.  But the theorems are fundamental, describing the long-term growth behavior of a wide range of coupled linear systems, and the stability of many nonlinear systems.

Karlin's theorems apply to discrete time and discrete space systems in which a growth phase is followed by a movement phase, and these are iterated.  This paper focuses on one of the theorems, Theorem 5.2, and extends it to apply to continuous time and discrete space systems.  Extensions to continuous time and continuous space systems entail greater technicalities and are deferred to elsewhere.

Karlin's Theorem 5.2 states simply that in a system of objects undergoing exponential growth or decay at different rates in different sites, \emph{the greater the level of mixing of objects between sites, the lower the long-term growth rate of the whole aggregation of objects}.  Here is the theorem, where $\rho(\M)$ is the spectral radius of matrix $\M$, the largest magnitude of any eigenvalue of $\M$:
\begin{TheoremFiveTwo}[\citealt{Karlin:1982}]
Let $\M$ be a general nonnegative irreducible backward migration matrix.  Consider the family of migration matrices,
\[
\M(\alpha) = (1-\alpha) \I + \alpha \M.
\]
Then for any set of positive fitness values $\D = \diag{d_1, d_2, \ldots, d_n}$,
\[
\rho(\M(\alpha) \D)
\]
is decreasing as $\alpha$ increases (strictly, provided $\D \neq d \I$).
\end{TheoremFiveTwo}
While the result is cast in terms of the specific context, it should be understood that $\M$ can be any irreducible stochastic matrix, and $\D$ any positive diagonal matrix.

Karlin used the theorem to analyze the stability of coupled nonlinear systems, where each site $i$ has a continuous, differentiable map $f_i(x): [0,1] \mapsto [0,1]$, so the coupled system is:
\[
x_i(t+1) = \sum_j M_{ij} f_j(x_j(t))
\]
In the population genetics context, $f_i(0) = 0$, and the linearized stability dynamics for small $\x$ are
\eb\label{eq:MDsolution}
\x(t) = ( \M(\alpha) \ \D )^t \x(0),
\ee
where the diagonal elements of $\D$ are $D_i = \df{f_i(0)}{x}$.  The zero solution $\x(t)=0$ is unstable to perturbation if $\rho(\M\D) > 1$ and stable to perturbation if $\rho(\M\D) < 1$.

The generality of this result is already evident in that no assumptions are made on $\M$ beyond that it be stochastic, and irreducible, which means there is a path of non-zero elements $M_{i k_1}, M_{k_1 k_2}, \ldots, M_{k_c j}$ between any $i$ and any $j \neq i$.  And no assumptions are placed on the set of nonlinear $f_i(x)$ other than that they be differentiable, and their domain and range be the unit interval, and $f_i(0) = 0$.

The first use of this theorem outside its original context was to analyze the evolution of genetic transmission \citep{Altenberg:1984,Altenberg:and:Feldman:1987,Altenberg:2009:Linear}.  There, instead of objects moving between sites, the objects are  genomes transforming between genotypes.  Theorem 5.2 translates to the \emph{reduction principle}:  the zero solution for a gene that controls `mixing' between genotypes is unstable to perturbation (i.e. to introduction of the gene to the population) if the gene reduces mixing.

The use of this theorem in additional contexts, has to my knowledge, not yet occurred.  This may be due, perhaps, to the limitation of the assumptions of discrete time and discrete space, or simply due to the small community of theoretical population geneticists familiar with it.
\begin{table}[h!t!bp]
\begin{tabular*}{\columnwidth}{p{68pt} p{90pt} p{164pt}}
%\ \\
\hline {\bf Objects:}	& {\bf Sites:} &{\bf Site specific growth/decay rates:} \\
\hline genes	& habitats & fitnesses \\
\hline genomes	& genotypes & fitnesses \\
\hline organisms	& habitats & species survival and reproduction rates\\
\hline agricultural pests	& fields under treatment & replication and survival\\ 
\hline microbes	& tissue compartments & survival and replication rates\\
\hline metabolites	& tissue compartments & catabolic rates \\
\hline reactants	& reaction-diffusion medium &  reaction rates \\
\hline wastes	& reactors & waste breakdown rates\\
\hline photons 	& media & absorption rates\\
\hline particles 	& heterogenous matter & interaction and decay rates\\
\hline capital 	& investments & rates of return\\
\hline 
\end{tabular*}
\caption{A short list of systems exhibiting site-specific growth or decay and mixing between sites.}\label{Table:List}
\end{table}

One may ask whether there is anything about discrete time and space that is essential to the result, or whether there is a more general phenomenon that may extend to continuous time and space.  Here I show that the result can be extended to continuous time and discrete space, namely, to coupled linear differential equations.  The result here applies to any combination of constant exponential growth or decay rates at different sites, and any constant pattern of movement between sites.  The extension of Karlin's theory to linear first order differential equations brings a much wider domain of systems into its purview.  One can contemplate a variety of systems that contain the applicable ingredients, shown in Table \ref{Table:List}.

%%%%%%%%%%%%%%%%%%%%%%%%%%%%%%%%%%%%%
\section{The System}

The system investigated here is of the form 
\eb
\df{\x(t)}{t} = (\Grm + m \Mvm)  \ \x(t),\label{eq:DiffEq}
\ee
which has solution
\eb
\x(t) = e^{(\Grm + m \Mvm)\, t}\ \x(0), \label{eq:DiffEqSolution}
\ee
where
\bd
\item[$\x(t) \in \Reals^n$] is an $n$-long vector of the quantities in each site at time $t$, 
\item[$\Grm$] is a real  $n \times n$ diagonal matrix of the growth rates at each site,
\item[$m \geq 0$] is the global rate of mixing between sites, and
\item[$\Mvm$] is a real $n \times n$ matrix that represents the movement distribution among $n$ sites. $\A$ is an \emph{essentially non-negative} matrix (also called a \emph{Metzler}, \emph{Metzler-Leontief}, or \emph{ML-matrix}), defined by $M_{ij} \geq 0$ for $i \neq j$.
\ed

The matrices $\Grm$ and $\Mvm$ hold the specifics for a particular system.  The diagonal matrix $\Grm$ contains the growth rates, and under the action of $\Grm$ alone, the time trajectories \eqref{eq:DiffEqSolution} would be 
\aligno{
x_i(t) = e^{\Gr_i t} x_i(0).
}
Exponential growth or decay is determined by whether $\Gr_i > 0$ or $\Gr_i < 0$.

The matrix $\Mvm$ represents movement between sites (or transformations between states).  The form of variation in movement examined here is of the form:
\eb\label{eq:MixVariation}
\Am(m) = \Grm + m \ \Mvm,
\ee
where $\Mvm$ represents the \emph{distribution} of movement, while $m$ represents the \emph{rate} of movement.   The question I address here is how the global level of `mixing', $m$, affects the asymptotic growth rates of \eqref{eq:DiffEqSolution}, and in particular, the stability of the zero solution $\x(t) = 0$ to perturbations.

The form \eqref{eq:MixVariation} includes, as a special case, the form considered by \citet{Karlin:1982}, in which a period of growth is followed by a period of movement: 
\eb\label{eq:KarlinsCase}
\Am(m)= [(1-m) \I + m \P ] \, \Grm = \Grm + m (\P \Grm -  \Grm),
\ee
where 
\bd
\item[$\I$] is the identity matrix, and 
\item[$\P$] is a stochastic matrix.
\ed
In continuous time, both growth and movement occur simultaneously, so the analog to \eqref{eq:KarlinsCase} is:
\eb\label{eq:NewCase}
\Am(m) = \Grm + m (\P - \I).
\ee

A typical assumption about movement is that quantity is redistributed but conserved, in which case summing the effect of movement over all destinations produces zero net change:  
\eb\label{eq:Conservation}
\ev^\tp \Mvm = \ev^\tp (\P \Grm -  \Grm) = \0, \mbox{\ and \ } \ev^\tp (\P - \I) = \0,
\ee
where 
\bd
\item[$\ev$] is the vectors of ones, and 
\item[$\scriptstyle \tp$] is the transpose of the vector or matrix.  
\ed
This class includes the generator matrices of continuous time Markov chains (also called `intensity' matrices).

If quantity is lost during movement, \eqref{eq:Conservation} is replaced by 
\eb\label{eq:Loss}
\ev^\tp \Mvm = \ev^\tp (\P \Grm -  \Grm) \leq \neq \0, \mbox{\ and \ } \ev^\tp (\P - \I) \leq \neq  \0.
\ee

%%%%%%%%%%%%%%%%%%%%%%%%%%%%%%%%%%%%
\section{Results}

The general phenomenon to be shown is the following:
\begin{Main}The asymptotic growth rate of
\[
\x(t) = e^{(\Grm + m \Mvm)\, t}\ \x(0),
\]
decreases with increasing values of the mixing parameter, $m$.  If this system exhibits net growth, then greater mixing inhibits the rate of growth.  If the system exhibits net decay, then greater mixing enhances the rate of decay.
\end{Main}

The asymptotic rate of growth or decay in \eqref{eq:DiffEqSolution} is given by the \emph{spectral abscissa} of $\Am(m)$, which is the largest real part of any eigenvalue of $\Am(m)$ (\citealt[p. 734]{Bernstein:2009}, \citealt[pp. 125-129]{Gantmacher:1959vol1}).  For real matrices that have non-negative off-diagonal elements (ML-matrices), the spectral abscissa is always an eigenvalue, referred to at the \emph{Perron root}. Irreducible ML-matrices retain many of the properties of irreducible non-negative matrices, including the existence of positive left and right eigenvectors (the Perron vectors), unique up to scaling, that are associated with the spectral abscissa (the Perron root) \citep[pp. 46--47]{Seneta:1981}.

This paper considers how variation in the global mixing rate $m$ varies the Perron root of $\Am(m)$.

\begin{Definitions}
Let:
\bd
\item [$\ev_i$] be the vector with element $i$ equal to $1$, and other elements equal to zero;
\item [$\pfe{\M} \eqdef \max_i \Realpart{\lambda_i(\M)}$] refer the spectral abscissa of a square matrix $\M$;
\item [$\vvf(\A) > 0$] refer to the right Perron vector of an irreducible ML-matrix $\A$, normalized so that $\ev^\tp \vvf(\A) = 1$; 
\item [$\uvf(\A)^\tp >0$] refer to the left Perron vector of an irreducible ML-matrix $\A$, normalized so that $\uvf(\A)^\tp \vvf(\A) = 1$ (the different normalization is convenient later);
So,
\aligno{
\uvf(\A)^\tp \A &= \pfe{\A} \ \uvf(\A)^\tp,\\
 \A \vvf(\A) &= \pfe{\A} \ \vvf(\A),  \mbox{\ and \ }\\
  \uvf(\A)^\tp \A \ \vv(\A) &= \pfe{\A}.
}
\ed
\end{Definitions}

%%%%%%%%%%%%%%%%%%%%%%%%%%%%%%%%%
\subsection{Basic Results}
The main result that will be used in the analysis is the following: 
%%%%%%%%%%%%%%%%
\begin{Theorem}\label{Theorem:Basic}
Let $\A$ be an $n \times n$ irreducible real matrix with non-negative off-diagonal elements (an ML-matrix), and $\D$ be an $n \times n$ diagonal real matrix.  Then:
\alignit{ \label{eq:Basic}
\pfe {\A + \D} - \pfe{\A}  \leq  \uvf(\A + \D)^\tp  \ \D \ \ \vvf(\A + \D)  
}
with equality if and only if $\D = c \ \I$ for some $c \in \Reals$.
\end{Theorem}
%%%%
\begin{proof}
The spectral abscissa of an irreducible ML-matrix, $\A $, is its Perron root, which is given by this variational formula (\citealt[Corollary 3.1]{Friedland:1981}, related to the variational formula of \citet{Donsker:and:Varadhan:1975}, and shown to extend to ML-matrices in \citealt[Lemma 3]{Altenberg:2009:Linear}):
\eb \label{eq:Friedland}
\pfe { \A } =  \sup_{\pv \in \Pc_n} \inf_{\x > \0} \sum_{i=1}^n p_i  \frac{[\A \x]_i}{x_i},
\ee
where $\Pc_n = \{\p \suchthat p_i \geq 0, \sum_{i=1}^n p_i = 1\} \subset \Reals^n$ .  

Let $\x(\A)$ and $\p(\A)$ be the vectors, as functions of $\A$, for which the $\sup$ and $\inf$ are attained, where $\x(\A)$ is also normalized so that $\sum_{i=1}^n x_i(\A) = 1$.  Then $\p (\A) $ and $\x (\A) $ are unique critical points for a given $\A$ (\citealt{Friedland:and:Karlin:1975}, \citealt{Friedland:1981}, and \citealt[p. 195]{Karlin:1982}),
\alignit{
\x (\A) &= \vvf(\A), \label{eq:X=V}
\intertext{and} 
\p(\A) &= \uvf(\A) \circ \vvf(\A) \label{eq:P=UV},
}
where $\circ$ is element-wise the Schur-Hadamard product. 

As utilized in the proof in \citet[Theorem 5.2]{Karlin:1982}, since $\x(\A)$ is a unique critical point in \eqref{eq:Friedland}, the $\inf$ means that any $\y \neq \x(\A)$ produces:
\alignit{\label{eq:Inf}
\pfe { \A }
 =  & \sup_{\pv \in \Pc_n} \inf_{\x > \0} \sum_{i=1}^n p_i  \frac{[\A \x]_i}{x_i} \notag \\
&< \sum_{i=1}^n p_i(\A)  \frac{[\A \y]_i}{y_i}.
}

Repeating the analogous step in the proof in \citet[Theorem 5.2]{Karlin:1982}, let $\y = \vvf(\A)$.  Then 
\aligno{
\sum_{i=1}^n &p_i(\A+\D)  \frac{[(\A+ \D) \vvf(\A)]_i}{v_i(\A)} \\
&= \pfe{\A} \sum_{i=1}^n p_i(\A+\D)  \frac {v_i(\A)}{v_i(\A)}
+  \sum_{i=1}^n p_i(\A+\D)  \frac{[ \D \vvf(\A)]_i}{v_i(\A)} \\
&= \pfe{\A}  +  \sum_{i=1}^n p_i(\A+\D)  D_i \\
&= \pfe{\A} + \uvf(\A+\D)^\tp \D \, \vvf(\A+\D).
}
Hence
\aligno{
\pfe{\A} + \uvf(\A+\D)^\tp \D \, \vvf(\A+\D) \geq \pfe{\A+\D},
}
with equality if and only if $\vvf(\A) = \vvf(\A+\D)$, which entails 
\aligno{
(\A+\D) \vvf(\A) &= \pfe{\A+\D} \vvf(\A) = \A\vvf(\A)+\D \vvf(\A) \\
&= \pfe{\A} \vvf(\A) + \D \vvf(\A),
}
hence $[ \pfe{\A+\D} -  \pfe{\A} ] \I \vvf(\A) = \D \vvf(\A)$;  and since $\vvf(\A) > \0$, this implies $ \D  = c \ \I$, where $c =  \pfe{\A+\D} -  \pfe{\A}$.
\end{proof}

From Theorem \ref{Theorem:Basic} a number of ancillary results can be obtained:
%%%%%%%%%%%%%%%
\begin{Corollary} [Convexity Derived] 
\label{Corollary:ConvexRelated}
Let $\A$ be an irreducible ML-matrix, and $\D$ a real diagonal matrix.  Then, for $\beta > 1$:
\alignit{ \label{eq:ConvexRelated}
\pfe { \A + \D} - \pfe{ \A} 
\leq \pfe{\frac{1}{\beta} \A + \D} - \pfe{\frac{1}{\beta}\A} 
\leq \uvf(\A + \beta \D)^\tp  \ \D \ \ \vvf (\A + \beta\D) ,
}
with equality if and only if $\D = c \ \I$ for some $c \in \Reals$.
\end {Corollary}
\begin{proof}
\citet{Cohen:1979:MPCBS} established the convexity relation:
\alignit{\label{eq:Cohen}
(1-\alpha) \pfe{\A} + \alpha \pfe{\A + \D} \geq \pfe{ (1-\alpha) \A + \alpha (\A + \D)}
}
for non-negative $\A$ and $\D$, and $0 < \alpha < 1$.   \citet{Friedland:1981} showed that when $\A$ is irreducible, equality holds if and only if $\D = c \ \I$ for some $c \in \Reals$.  This holds when $\A$ and $\D$ are ML-matrices by the relation $\pfe{\A + c \I} - c = \pfe{\A}$, since large enough $c$ will guarantee $\A + c \I \geq \0$.

Rearrangement of \eqref{eq:Cohen} gives: 
\aligno{
\pfe{\A + \D}& - \pfe{\A} \geq \\
& \frac{1}{\alpha} \left [ \pfe{ (1-\alpha) \A + \alpha (\A + \D) } -  \pfe{\A} \right ] \\
&=  \left [ \pfe{\frac{1}{\alpha}  \A + \D } -  \pfe{\frac{1}{\alpha}\A} \right ].
}
Application of \eqref{eq:Basic} gives
\aligno{
\uv(\A+\D)^\tp \D \vv(\A+\D) 
\geq  \pfe{\A + \D} - \pfe{\A} 
\geq  \pfe{\frac{1}{\alpha}  \A + \D } -  \pfe{\frac{1}{\alpha}\A}. 
}
The condition for equality in both places is that $\D = c\ \I$ for some $c \in \Reals$.  

Multiplying by $\alpha$, letting $\D' \eqdef \alpha \D$, and $\beta \eqdef 1/\alpha$:
\aligno{
&\uvf(\A + \D)^\tp  \ \alpha \D \ \ \vvf(\A + \D) 
\geq  \alpha \pfe{\A + \D} - \alpha \pfe{\A} 
\geq \pfe { \A + \alpha \D} - \pfe{ \A} \\
& \iff \\
&\uvf(\A + \beta \D')^\tp  \ \D' \ \ \vvf (\A + \beta \D') 
\geq  \pfe{\frac{1}{\beta} \A + \D'} - \pfe{\frac{1}{\beta}\A} 
\geq \pfe { \A + \D'} - \pfe{ \A}.
}
Replacement of $\D'$ by $\D$ gives \eqref{eq:ConvexRelated}.
\end{proof}

%%%%%%%%
\begin{Corollary}\label{Corollary:Basic}
For irreducible ML-matrix $\A$ and real diagonal matrix $\D$:
\alignit{\label{eq:CorollaryBasic}
\uvf(\A + \D)^\tp \A \ \vvf(\A + \D) \leq \pfe{\A},
}
with equality if and only if $\D = c \ \I$ for some $c \in \Reals$.
\end{Corollary}
%%%%
\begin{proof}
From Theorem \ref {Theorem:Basic}:
\aligno{
\uvf(\A + \D)^\tp& \D \ \vvf(\A + \D)    \\
\geq & \pfe {{\A + \D}} - \pfe{\A} \\
&= \uvf(\A + \D)^\tp \ (\A+\D) \ \vvf(\A + \D) - \pfe{\A} \\
&= \uvf(\A + \D)^\tp \A \ \vvf(\A + \D) \\
&\qquad + \uvf(\A + \D)^\tp \D \ \vvf(\A + \D) - \pfe{\A},
\intertext{and rearranging,}
\pfe{\A}  \geq & \  \uvf(\A + \D)^\tp \A \ \vvf(\A + \D).
}
The equality condition is unchanged from Theorem \ref{Theorem:Basic}.
\end{proof}

%%%%%%%%%%%%%%%
\begin{Corollary}[Sums]\label{Corollary:Sums}
Let $\D$ be a real diagonal matrix, and let 
$
\A =  \sum_{k=1}^N \Mvm_k,
$
where $\{\Mvm_k \}$ are ML-matrices that share a common right [left] Perron vector.  Then
\eb\label{eq:CorollarySums}
 \uv(\A+\D)^\tp \A \ \vv(\A+\D) 
\leq \sum_{k=1}^N \pfe{\Mvm_k},
\ee
with equality if and only if $\D = c \ \I$ for some $c \in \Reals$.
\end{Corollary} 
%%%%%%%%
\begin{proof}
Letting $\y$ be the right common Perron vector, then 
\[
\A \y = \pfe{\A} \y = \sum_{k=1}^N \Mvm_i \y = \y \sum_{k=1}^N \pfe{\Mvm_i},
\]
thus $\pfe{\A} =  \sum_{k=1}^N \pfe{\Mvm_i}$, and substitution in \eqref{eq:CorollaryBasic} yields \eqref{eq:CorollarySums}.  For $\y$ equal to a common left Perron vector, $\A^\tp$ is used.
\end{proof}

%%%%%%%%%%%%%%%
\begin{Corollary} [The `Flip' Theorem \protect{\citep[Theorem 3.2.5]{Bapat:and:Raghavan:1997}}] 
\label{Corollary:Flip}
Let $\A$ be an irreducible ML-matrix.  Let ${\y \circ \z = \uv(\A) \circ \vv(\A)}$, with $\y > \0$, $\z > \0$.  Then
\aligno{
\z^\tp  \A \ \y \geq \uvf(\A)^\tp  \A  \ \vvf(\A) = \pfe{\A}. 
}
with equality if and only if $\y = \v(\A)$ or $\uv(\A) = \v(\A)$.

In particular, when the left and right Perron vectors are `flipped',
\aligno{
\vvf(\A)^\tp   \A  \ \uvf(\A) > \uvf(\A)^\tp   \A  \ \vvf(\A) = \pfe{\A}, 
}
if $\uv(\A) \neq \v(\A)$.
\end {Corollary}
\begin{proof}
This is an alternative proof to that given in \citet {Bapat:and:Raghavan:1997}, and this extends Theorem 3.2.5 to ML-matrices.  

Substituting $y_i = u_i(\A) v_i(\A) / z_i$ and \eqref{eq:P=UV} into \eqref{eq:Inf}, one gets:
\aligno{
\pfe { \A } \leq & \sum_{i=1}^n p_i(\A)  \frac{[\A \y]_i}{y_i}
= \sum_{i=1}^n u_i(\A) v_i(\A)  \frac{[\A \y ]_i}{u_i(\A) v_i(\A) / z_i}  \\
& =  \sum_{i=1}^n z_i [\A \ y ]_i = \z^\tp \A \ \y,
}
with equality if and only if $\y= \vv(\A)$ or $\uv(\A) = \vv(\A)$.
\end{proof}

%%%%%%%%%%%%%%%%%%%%%%%%
\subsection{Main Result}
These results are now applied to extend Karlin's Theorem 5.2 \citep{Karlin:1982}.  Theorem 5.2 applies to matrices $\Am(m) = \Grm + m \ \Mvm$, where $\Grm$ is positive diagonal matrix, $\Mvm = (\P - \I) \Grm$, $\P$ is an irreducible stochastic matrix, and $0 \leq m \leq 1$.  Here, results are extended to $\Grm$ that may have negative diagonal elements, to arbitrary irreducible and reducible ML-matrices, $\Mvm$, and to any $m \geq 0$.

%%%%%%%%%%%%%%%%
\begin{Theorem}[Growth and Mixing] \label{Theorem:Main}
Let $\Grm$ be a real $n \times n$ diagonal matrix, and $\Mvm$ be a real $n \times n$ matrix with non-negative off-diagonal elements (an ML-matrix).  Then, for  $m \geq 0$:
\eb \label{eq:Main}
\df{ \pfe{ \Grm + m \Mvm }}{m} \leq  \pfe{\Mvm},
\ee
with equality holding if and only if either:
\benu
\item $\Grm = c \ \I$ for some $c \in \Reals$; or,
\item 
	\benu
	\item $\Mvm$ is reducible; and
	\item \label{GcIk} $\Grm_{\kappa} = c \ \I_{\kappa}$, for some $c \in \Reals$, for every $\kappa$ such that: 
		\benu
		\item $\kappa$ is an index on the  diagonal blocks in the Frobenius normal form of $\Mvm$, 
each block being defined by a subset of the indices of $\Mvm$, 
		\item $\Grm_\kappa \mbox{\it \ and } \I_\kappa$ are the restrictions of $\Grm$ and $\I$ under the block of indices derived from $\Mvm$, and
		\item \label{item:kappa} on an open neighborhood of $m$:  
\aligno{
\pfe{\Grm_{\kappa} + m \Mvm_{\kappa}}&= {max_{h}} \ \pfe{\Grm_h + m \Mvm_h}
= \pfe{\Grm + m \Mvm} .
}
		\eenu
	\eenu
\eenu
Moreover, for reducible $\Mvm$, with $\kappa$ defined as in \ref{item:kappa}, a sharper inequality obtains:
\[
\df{ \pfe {\Grm + m \Mvm}}{m} \leq  \pfe{\Mvm_{\kappa}} \leq  \pfe{\Mvm},
\]
with equality on the left side only under condition \ref{GcIk}.
\end{Theorem}
%%%%%%%%

\begin{proof}
{\bf Case: Irreducible $\Mvm$.}
When $\Am(m)$ is a $C^2$ function of $m$, and is an irreducible ML-matrix on some open set around $m$, then the derivative of its spectral abscissa is \citep[Lemma 4]{Altenberg:2009:Linear}:
\aligno{
\df{ \pfe { \Am(m)} }{m} &= \sum_{i=1}^n p_i(m)  \frac{[ \df{\Am(m)}{m} \; \x(m)]_i}{x_i(m)}, \\
}
where for clarity, $\p(m) \eqdef \p(\Am(m))$ and $\x(m) \eqdef \x(\Am(m))$.  Using \eqref{eq:X=V} and \eqref{eq:P=UV} one obtains the classical form \citep[Sec. 9.1.1]{Caswell:2000}:
\aligno{
\df{ \pfe { \Am(m)} }{m} &= \uv(m)^\tp \df{\Am(m)}{m} \vv(m).
}
Hence, 
\aligno{
\df{ \pfe {\Grm + m \ \Mvm )} }{m} 
= \uv(m)^\tp \Mvm \ \vv(m).
}
Applying Corollary \ref{Corollary:Basic}:
\aligno{
 \uvf(m)^\tp (m \  \Mvm) \ \vvf(m) & \leq \pfe{m \ \Mvm}\\
\iff \uvf(m)^\tp  \Mvm \ \vvf(m) &\leq  \pfe{\Mvm},
\intertext{and thus}
\df{ \pfe{ \Grm + m \Mvm }}{m} &= \uvf(m)^\tp \Mvm \ \vvf(m) \leq  \pfe{\Mvm},
}
with equality if and only if $\Grm = c \ \I$ for some $c \in \Reals$.  

%%%%%%%%%%%%
{\bf Case: Reducible $\Mvm$.}
If $\Mvm$ is reducible, its Frobenius normal form is utilized.   The Frobenius normal form, $\bar{\A}$, of a reducible matrix $\A$  permutes the indices until it has the structure \citep[p. 75]{Gantmacher:1959vol2}:
\alignit{\label{eq:FrobeniusNormalForm}
\bar{\A} = 
%%%&\bar{\A} = \\
&\matrx{
\begin{array}{cccc|ccc}
\A_1 & \0 & \cdots & \0 & & & \\
\0 & \A_2 & \ddots & \vdots & &  \mbox{\Large \bf 0} & \\
\vdots &  & \ddots & \0 & &  & \\
\0 & \cdots & \0& \A_t & & & \\
\hline \A_{t+1,1} & \A_{t+1, 2} & \cdots & \A_{t+1,t} & \A_{t+1}  & \0 & \0\\
\vdots & \vdots & \cdots & \cdots & \cdots & \ddots & \ \0\\
\A_{t+s,1} & \A_{t+s, 2} & \cdots & \A_{t+s, t} & \A_{t+s, t+1} &\cdots &\A_{t+s}
\end{array}}
}
where the diagonal blocks\renewcommand{\A}{{\bm{A}}}
 $\A_h$ are irreducible square matrices.  
The eigenvalues of $\A$ are the eigenvalues of the irreducible diagonal block matrices $\A_h$
\citep[Lemma 5]{Altenberg:2009:Linear}.  Therefore, the spectral abscissa for $\Am(m) = \Grm + m \ \Mvm$ is the maximum of the spectral abscissae:
\[
\pfe{ \Am(m)}  = \max_{h} \pfe{ \Am_h(m)}.
\]
Hence, for every $\kappa$ that satisfies $\pfe{ \Am_{\kappa}(m )} = \max_{h} \pfe{ \Am_h(m)}$ on some neighborhood of $m$, the result on irreducible matrices yields:
\eb\label{eq:block}
\df{}{m} \pfe{ \Am(m)}  =  \df{}{m}  \pfe{ \Am_{\kappa}(m)}  \leq \pfe{\Mvm_{\kappa}},
\ee
with equality holding if and only if $\Grm_{\kappa} = c \ \I_{\kappa}$ for some $c \in \Reals$ for every $\kappa$.

Since $\pfe{\Mvm} = \max_h \pfe{\Mvm_h} \geq \pfe{\Mvm_{\kappa}}$, \eqref{eq:block} is seen to be a sharper inequality than \eqref {eq:Main}:
\[
\df{}{m} \pfe{ \Am(m)}  =  \df{}{m}  \pfe{ \Am_{\kappa}(m)}  \leq \pfe{\Mvm_{\kappa}} \leq \pfe{\Mvm}. \qedhere
\]
\end{proof}
{\bf Remark}:  Blocks $\A_1, \ldots, \A_t$ in \eqref{eq:FrobeniusNormalForm} are called \emph{isolated blocks}.  If $\x(0)$ is zero on any isolated block, it remains zero on that block for all $t\geq 0$.  Should $\pfe{\A}$ derive from isolated blocks, then the system will asymptotically grow at rate $\pfe{\A}$ only if $\x(0)$ has non-zero elements for one of those blocks.  Therefore, the asymptotic growth rate for reducible $\A$ may depend on the initial value $\x(0)$, whereas for irreducible $\A$ it is independent of any initial value $\x(0) \geq \neq \0$.

\subsection{Conservative and Lossy Mixing}
Inherent to the concepts of movement and mixing is the idea that the redistributed quantity is conserved or perhaps lost by the movement process, but never created.  Theorem \ref{Theorem:Main} is more general and does not assume this.  But when the assumption that movement is conservative or lossy is made, it yields the following results as special cases of Theorem \ref{Theorem:Main}:

%%%%%%%%%%%%%%%
\begin{Corollary} [Measure-Preserving Mixing] 
\label{Corollary:Preserving}
If ML-matrix $\A$ is such that mixing preserves the total measure of quantity, then
\aligno{
\df{ \pfe{ \Grm + m \Mvm }}{m} \leq 0,
}
with equality under the same conditions as Theorem \ref{Theorem:Main}.
\end {Corollary}
\begin{proof}
Measure-preserving mixing means $\ev^\tp \Mvm = \0$, in which case $ \pfe{\Mvm} = 0$, and this is applied to \eqref{eq:Main}.
\end{proof}

%%%%%%%%%%%%%%%
\begin{Corollary} [Lossy Mixing] 
\label{Corollary:Lossy}
If ML-matrix $\A$ is such that quantity is lost in mixing,  then
\aligno{
\df{ \pfe{ \Grm + m \Mvm }}{m} < 0,
}
for any real diagonal matrix $\Grm$.
\end {Corollary}
\begin{proof}
Lossy mixing means $\ev^\tp \Mvm \leq \neq \0$, which implies by subinvariance \citep[Corollary 3, p. 52]{Seneta:1981}  that $ \pfe{\Mvm} < 0$, and this is applied to \eqref{eq:Main}.  
\end{proof}

%%%%%%%%%%%%%
\begin{Lemma}[Bounds]\label{Lemma:Bounds}
Let $\pfe{\Mvm} = 0$. Then:
\[
 \pfe{\Grm + m \Mvm} \in [ \min_i (\Gr_i), \max_i (\Gr_i)]
\]
\end{Lemma}
%%%%%%%%
\begin{proof}
We know from \citet{Cohen:1978} that $\df{\pfe{\A} }{A_{ij}} \geq 0$ for any ML-matrix $\A$.  So here,
$\df{}{\Gr_i} \pfe{\Grm + m \Mvm} \geq 0$.  Hence
\aligno{
\pfe{\max_i \Gr_i \I + m \Mvm} &= \max_i \Gr_i + m \pfe{\Mvm} = \max_i \Gr_i \\
& \geq  \pfe{\Grm + m \Mvm} \geq  \pfe{\min_i \Gr_i \I + m \Mvm}\\
& \qquad = \min_i \Gr_i.  \qquad \qedhere
}
\end{proof}

%%%%%%%%%%%%%
\begin{Theorem}[Limit]\label{Theorem:Limit}
For any choice of $\lambda \in [\max_i (\Gr_i), \min_i (\Gr_i)]$, there exists a family of ML-matrices, $\Fc \subset \{\Mvm \suchthat \pfe{\Mvm} = 0\}$, that yields
\[
\lim_{m \goesto \infty} \pfe{\Grm + m \ \Mvm} = \lambda \mbox{\ for } \Mvm \in \Fc.
\]
One such family is:
\[
\Fc = \{  \alphav \ev^\tp - \I \suchthat \sum_{i=1}^n \alpha_i \Gr_i = \lambda, \ \alpha_i \geq 0, \mbox{ and } \ev^\tp \alphav= 1 \}.
\]
\end{Theorem}
%%%%
\begin{proof}
First, let $\alphav > \0$, which makes matrices $\alphav \ev^\tp$ and $\Am(m) \eqdef \Grm + m \ \alphav \ev^\tp$ irreducible, giving $\Am(m)$ a unique positive right Perron vector, $\vvf(m)$.  Hence
%%%%
\alignit{
\label{eq:HoC}
\pfe{\Am(m)} \ \vvf(m) &= \Grm\vvf(m) + m \, \alphav \, \ev^\tp \vvf(m)- m \, \vvf(m) \notag \\
& =   \Grm \vvf(m) + m (\alphav - \vvf(m)) \\
& \iff \notag \\
v_i(m) = &\frac{m \, \alpha_i}{ \pfe{\Am(m)} - \Gr_i + m} =  \frac{\alpha_i}{\dspfrac{\pfe{\Am(m)} - \Grm}{m} + 1} \notag
}
By Lemma \ref{Lemma:Bounds}, $\pfe{\Am(m)}$ is bounded, hence 
\aligno{
\lim_{m \goesto \infty} v_i(m) = \lim_{m \goesto \infty} \frac{\alpha_i}{\dspfrac{\pfe{\Am(m)} - \Grm}{m} + 1} = \alpha_i.
}
When the elements in \eqref{eq:HoC} are summed:
\aligno{
\pfe{\Am(m)} \ \ev^\tp \vvf(m) = \pfe{\Am(m)}  &=  \ev^\tp  \Grm \vvf(m) + m \, \ev^\tp (\alphav - \vvf(m))\\
& =  \ev^\tp  \Grm \vvf(m) + m (1 - 1) \\
&  =  \ev^\tp  \Grm \vvf(m).
}
Hence, 
$$
\lim_{m \goesto \infty} \pfe{\Am(m)}  
= \lim_{m \goesto \infty} \ev^\tp  \Grm \vvf(m)
=  \ev^\tp  \Grm \alphav = \sum_{i=1}^n \Gr_i \alpha_i.
$$

The assumption that $\alphav > \0$ allows some solution to $\sum_{i=1}^n \Gr_i \alpha_i = \lambda$ for any $\lambda \in ( \max_i (\Gr_i), \min_i (\Gr_i) )$.  But for $\lambda = \Gr_{i_{\max}}$ or $\lambda = \Gr_{i_{\min}}$, where $i_{\max}$ and $i_{\min}$ are the indices producing $\max_i (\Gr_i)$ and $ \min_i (\Gr_i)$, then $\alphav = \ev_{i_{\max}}$ or $\alphav = \ev_{i_{\min}}$, respectively.  In these cases, $\Grm + m \ \alphav \ev^\tp$ is no longer irreducible.  The Frobenius normal form for $\Am(m) = \Grm + m( \ev_i \ev^\tp - \I)$ is
\aligno{
\bar{\Am}(m) = \matrx{
\begin{array}{cccccc}
\Gr_1-m &  & & &\BigZero \\
& \Gr_2-m &  && \\
&   & \ddots & &\\
\BigZero&  &  &  \Gr_j-m & &\\
& && & \ddots &\\
m & m & \cdots &m & \cdots & \Gr_i
\end{array}
}_.
}
Hence, the eigenvalues of $ \Grm + m( \ev_i \ev^\tp - \I)$ are the diagonal elements $\{\Gr_j - m \suchthat j \neq i \} \union \{\Gr_i\}$.  Thus when $m > \max_j ( \Gr_j -  \Gr_i)$ then $\pfe{\Am(m)} = \Gr_i > \Gr_j - m$ for all $j \neq i$.  So,
\aligno{
\lim_{m \goesto \infty} \pfe{\Grm + m( \ev_i \ev^\tp - \I) } =  \Gr_i.
}
Letting $i = i_{\max}$ and $i = i_{\min}$ completes the construction.
\end{proof}

%%%%%%%%%%%%%
\begin{Corollary}[Stability] \label{Corollary:Stability}
Let $\Am(m) = \Grm + m \ \Mvm$ as in Theorem \ref{Theorem:Main}.  Suppose that $\pfe{\Mvm} \leq 0$.  If the solution $\x(t) = 0$ is unstable under \eqref{eq:DiffEqSolution} for some $m^*$, then it is unstable for all $0 \leq m \leq m^*$.
\end{Corollary}
\begin{proof}
Instability means that $\pfe{\Am(m)} > 0$.  By Theorem \ref{Theorem:Main}, when $\pfe{\Mvm} \leq 0$,
$$
\df{ \pfe{\Am(m)}}{m} \leq  \pfe{\Mvm} \leq 0.
$$
 Hence, if $\pfe{\Am(m^*)} > 0$, then $\pfe{\Am(m)} \geq \pfe{\Am(m^*)} > 0$ for all $0 \leq m \leq m^*$.
\end{proof}

\subsection{Additional Results}
%%%%%%%%%%%%%%%
\begin{Corollary}[Heterogeneity]\label {Corollary:Heterogeneity}
Under low enough mixing, heterogeneity of growth rates always produces greater asymptotic growth than the average of the growth rates:
\aligno{
\exists m^* > 0 \suchthat  \ \forall \ 0 \leq m < m^*: \spab(\Grm + m \Mvm) > \frac{1}{n} \sum_{i=1}^n \Gr_i.
}
\end{Corollary}
\begin{proof}
When $\Grm \neq c \ \I$  for any $c \in \Reals$, then $\pfe{\Grm} = \max_i( \Gr_i) >  \frac{1}{n} \sum_{i=1}^n \Gr_i$.  Since the eigenvalues are continuous functions of the entries of $\Grm + m \ \Mvm$, there is some neighborhood  $(0, m^*)$ where $\spab(\Grm + m \Mvm) > \frac{1}{n} \sum_{i=1}^n \Gr_i$ for $m \in (0, m^*)$. 
\end{proof}

%%%%%%%%%%%%%%%
\begin{Theorem} [Convexity] 
\label{Theorem:Convexity}
Let $\Grm$ be a real diagonal matrix, and $\Mvm$ be an irreducible ML-matrix.  Then $\pfe{\Grm + m\ \Mvm}$ is convex in $m$. Specifically:

For any $m_1, m_2 \geq 0, m_1 \neq m_2$ and $0 < \alpha < 1$, 
one has:
\alignit{
\label{eq:mConvexity}
& \pfe{ (1-\alpha) \ (\Grm+ m_1 \Mvm ) + \alpha \ (\Grm+ m_2 \Mvm )} \notag \\
& \leq  (1-\alpha) \ \pfe{\Grm+ m_1 \Mvm} + \alpha \ \pfe{\Grm+ m_2 \Mvm },
}
with equality if and only if $\Grm = c \ \I$, for some $c \in \Reals$.
\end{Theorem}
%%%%%%%%
\begin{proof}
Convexity of $\pfe{\Grm + m\ \Mvm}$ with respect to diagonal matrix $\Grm$ was established by \citet[Theorem 3]{Cohen:1979:MPCBS}.  Specifically, for real diagonal matrices $\D_i$ and ML-matrix $\A$, for $0 < \beta < 1$:
\alignit{
\label{eq:DConvexity}
&\pfe{ (1-\beta) \D_1  +  \beta \D_2  + \A)} \notag \\
& \leq 
(1-\beta) \ \pfe{\D_1 +  \A} +  \beta \ \pfe{\D_2 +  \A}.
}
\citet[Theorem 4.1]{Friedland:1981} showed further that equality in \eqref{eq:DConvexity} obtains if and only if $\D_1 - \D_2 = c \ \I$ for some $c \in \Reals$.  

%%%%%%%%%%
Convexity with respect to the diagonal matrix implies convexity with respect to $m$, which can be seen by morphing  \eqref{eq:mConvexity} into \eqref{eq:DConvexity} through the identity
\[
\pfe{\Grm + m\ \Mvm} = m \ \pfe{\frac{1}{m} \Grm + \Mvm}\suchthat
\]
\eban
\lefteqn{
\pfe{ \Grm  + [ (1-\alpha)  m_1 + \alpha m_2] \Mvm) } } \\
&=& [(1-\alpha)  m_1 + \alpha m_2 ] \, \pfe{ \frac{1}{[ (1-\alpha)  m_1 + \alpha m_2]} \Grm + \Mvm} \\
&\leq&   (1-\alpha) \, \pfe{ \Grm + m_1 \Mvm} + \alpha \, \pfe{\Grm + m_2 \Mvm} \\
&&= (1-\alpha) m_1 \,  \pfe{  \frac{1}{m_1} \Grm + \Mvm} \\
&& \qquad +\  \alpha \, m_2\, \pfe{  \frac{1}{m_2} \Grm +\Mvm }.
\eean
Dividing both sides by $(1-\alpha)  m_1 + \alpha m_2$ gives:
\eba\label{eq:MainTheoremPr1}
\lefteqn{ \pfe{ \frac{1}{[ (1-\alpha)  m_1 + \alpha m_2]} \Grm + \Mvm} } \nonumber \\
&& \leq \frac{(1-\alpha)m_1}{(1-\alpha)  m_1 + \alpha m_2}  \,   \pfe{  \frac{1}{m_1} \Grm + \Mvm}  \nonumber \\
&& +  \frac{\alpha m_2}{(1-\alpha)  m_1 + \alpha m_2} \,  \pfe{  \frac{1}{m_2} \Grm +\Mvm }.
\eea
 Now, define
 \[
\D_1 := \frac{1}{m_1} \Grm, \, \, \D_2 :=  \frac{1}{m_2} \Grm,
 \]
and
\[
 \beta :=  \frac{\alpha m_2}{(1-\alpha)  m_1 + \alpha m_2}, \mbox{\ so \ } 1-\beta =  \frac{(1-\alpha) m_1}{(1-\alpha)  m_1 + \alpha m_2}.
\]
Thus \eqref{eq:MainTheoremPr1} becomes:
\alignit{
\label{eq:MainTheoremPr2}
&\pfe{ (1-\beta) \D_1  +  \beta\D_2  + \Mvm} \notag \\
&\leq 
 (1-\beta) \ \pfe{\D_1 +  \Mvm } +  \beta \ \pfe {\D_2 +  \Mvm },
}
which is \eqref{eq:DConvexity}.  The equality condition, $\D_1 - \D_2 = c \ \I$ for some $c \in \Reals$, becomes $\D_1 - \D_2 = (1/m_1 - 1/m_2) \Grm = c \ \I$.  Since $m_1 \neq m_2$, this is precluded if $\Grm \neq c \ \I$ for any $c \in \Reals$, in which case $\pfe{\Am(m)}$ is strictly convex in $m$.
\end{proof}

%%%%%%%%%%%%%%%%%%%%%%%%%%%%%%%%
\section{Discussion}

The proximate motivation for extending Karlin's Theorem 5.2 from maps to differential equations was a theoretical study by \citet{Steinmeyer:and:Wilke:2009} on the effect of tissue compartments for antiviral, lethal mutagenesis therapy.  The background of the problem is that a number of anti-viral agents seem to work by mutating the virus to inviability.  But different tissues can concentrate the therapeutic mutagen to different concentrations.  The greater the concentration, the lower that the viral replication rate becomes.  Thus the virus will have different growth rates in different compartments.  

\citet{Steinmeyer:and:Wilke:2009} ask how heterogeneity in mutagen concentrations, and movement of virions between compartments, affects the dosage needed to cure the infection.  They obtain analytic results for a two compartment model, but for more than two compartments must resort to high $m$ and low $m$ limits, course grained approximations, and numerical studies, to obtain results.

Their results exhibit a number of salient features.  Here it is noted where each of these features is proven analytically:  
\benu
\item Compartmental heterogeneity of growth rates increases long-term viral growth rates above the average of the compartment growth rates (Corollary \ref{Corollary:Heterogeneity});
\item Greater viral mixing between compartments decreases the long-term total viral growth rate (Theorem \ref{Theorem:Main});
\item The decrease in viral growth rate is convex in the amount of mixing (Theorem \ref{Theorem:Convexity});
\item  The whole body viral growth rate for all levels of mixing is bounded above by the maximum growth rate in any compartment (Lemma \ref{Lemma:Bounds});
\item As the amount mixing of virions between compartments increases, it converges to a weighted average of the compartmental growth rates (Theorem \ref{Theorem:Limit}).
\eenu
As can be seen, each of these qualitative features observed in the two-compartment and numerical results of \citet{Steinmeyer:and:Wilke:2009} are in fact properties that extend to any number of compartments, arbitrary growth rates among compartments, patterns of mixing between compartments, and magnitudes of mixing.

\subsection{Wider Application}

An intimated in Table \ref{Table:List}, the results here may find application in a diverse array of systems.  The model for viral replication in multiple tissue compartments in \linebreak\citet{Steinmeyer:and:Wilke:2009} has the same form as ecological models of sink and source populations \citep{Holt:1996:Oikos,Armsworth:and:Roughgarden:2005:AmNat}. 

In addition to linear  dynamical systems of the form \eqref{eq:DiffEq}, many nonlinear systems may have  \eqref{eq:DiffEq} as the dynamic of small perturbations.  It would be informative  to collect other examples of models of the form \eqref{eq:DiffEq} from the literature in the various fields mentioned in Table \ref{Table:List}.

\subsection{Further Extensions}

One of the important sources of ML-matrices is the numerical solution of second-order partial differential equations.  The second derivative can be discretely approximated by the centered difference method, which in one dimension generates a tri-diagonal ML-matrix \citep[p. 38-4]{Beattie:2007:HLA}, with $1$s along the super- and sub-diagonals, and $-2$ along the diagonal. Thus, a centered difference approximation to 
\alignit{\label{eq:2ndOrder}
\pf{x(s,t)}{t} = g(s) \ x(s,t) + m \ \ppf{x(s,t)}{s},
}
with the proper conditions should be of the form \eqref{eq:DiffEq}, and the results here would apply.  The second derivative in higher dimensions, such as the Laplace operator, also has discrete approximations that are ML-matrices (\citealt[pp. 40-1 -- 40-2]{Ng:2007:HLA}; \citealt[pp. 41-1 -- 41-2]{Greenbaum:2007:HLA}), so one can seek analogous results in higher dimensional diffusions.

Under the analytic assumptions for which the solutions of the centered difference approximation converge to positive eigenfunction solutions of \eqref{eq:2ndOrder}, one can expect that the Perron root will be a decreasing function of $m$.  

Here I have only touched upon some applications and extensions one might find for the results presented.  The fact that Karlin's result --- that mixing reduces growth --- extends from discrete-time/discrete space systems to continuous-time/discrete space systems, and promises to extend further to continuous time and space systems, and to diffusion operators, suggests something fundamental in the phenomenon.

%%%%%%%%%%%%%%%%%%%%%%%%%%%%%%%%%%%%
\section*{Acknowledgements}
I thank Shmuel Friedland for the invitation to contribute to the 16th International Linear Algebra Society Conference in Pisa, which facilitated the present work.

%%%%%%%%%%%%%%%%%%%%%%% 
%\input{/usr/local/LATEX/biblist.tex}

\end{document}